\documentclass{amsart}

\usepackage{amsmath,amssymb,amsthm}

\usepackage[nohug,heads=LaTeX]{diagrams}

\newtheorem{prop}{Proposition}[section]
\newtheorem{thm}[prop]{Theorem}
\newtheorem{lem}[prop]{Lemma}
\newtheorem{cor}[prop]{Corollary}

\theoremstyle{definition}

\newtheorem{ex}[prop]{Example}
\newtheorem{note}[prop]{Note}

\newtheorem{rem}[prop]{Remark}

\newtheorem*{ack}{Acknowledgements}

\def\co{\colon\thinspace}

\newcommand{\oalpha}{\overline{\alpha}}
\newcommand{\bfa}{\mathbf{a}}

\newcommand{\obeta}{\overline{\beta}}
\newcommand{\ob}{\overline{b}}

\newcommand{\C}{\mathbb{C}}
\newcommand{\Ca}{C_{\mathrm{aff}}}
\newcommand{\CP}{\mathbb{C}\mathrm{P}}
\newcommand{\tC}{\widetilde{C}}

\newcommand{\rme}{\mathrm{e}}
\newcommand{\ole}{\overline{e}}

\newcommand{\tF}{\widetilde{F}}

\newcommand{\rmi}{\mathrm{i}}

\newcommand{\tlambda}{\tilde{\lambda}}

\newcommand{\oM}{\overline{M}}
\newcommand{\tmu}{\tilde{\mu}}

\newcommand{\N}{\mathbb{N}}

\newcommand{\bbP}{\mathbb{P}}

\newcommand{\R}{\mathbb{R}}
\newcommand{\orho}{\overline{\rho}}

\newcommand{\oSigma}{\overline{\Sigma}}
\newcommand{\hSigma}{\hat{\Sigma}}
\newcommand{\tSigma}{\widetilde{\Sigma}}
\newcommand{\tsigma}{\tilde{\sigma}}

\newcommand{\tV}{\widetilde{V}}

\newcommand{\Z}{\mathbb{Z}}

\DeclareMathOperator{\Int}{Int}

\begin{document}

\author[B.~Albach]{Bernhard Albach}
\author[H.~Geiges]{Hansj\"org Geiges}
\address{Mathematisches Institut, Universit\"at zu K\"oln,
Weyertal 86--90, 50931 K\"oln, Germany}
\email{albachbernhard@gmail.com, geiges@math.uni-koeln.de}

\title{Surfaces of section for Seifert fibrations}

\date{}

\begin{abstract}
We classify global surfaces of section for
flows on $3$-manifolds defining Seifert fibrations.
We discuss branched coverings --- one way or the other ---
between surfaces of section for the Hopf flow and those
for any other Seifert fibration of the $3$-sphere, and we relate these
surfaces of section to algebraic curves in weighted
complex projective planes.
\end{abstract}


\thanks{This research is part of a project in the SFB/TRR 191
\textit{Symplectic Structures in Geometry, Algebra and Dynamics},
funded by the DFG (Project-ID 281071066 -- TRR 191)}

\maketitle


\section{Introduction}
Global surfaces of section are an important tool for understanding the
dynamics of non-singular vector fields on manifolds. For a given flow
or class of flows, it is a basic question to understand the existence
of such surfaces of section, and their topological properties.

Here are some recent results in this direction in the
context of Reeb dynamics; for a more
comprehensive overview of the literature we refer to~\cite{agz}
and~\cite{sahr18}.
Hryniewicz--Salom\~ao--Wysocki~\cite{hsw} establish sufficient
(and $C^{\infty}$-generically necessary) conditions for
a finite collection of periodic orbits of a Reeb flow on
a closed $3$-manifold to bound a positive global surface
of section of genus zero. Building on the work of
Hryniewicz~\cite{hryn14}, who found a characterisation of
the periodic Reeb orbits of dynamically convex contact forms on the
$3$-sphere that bound disc-like global surfaces of section,
existence results for higher genus surfaces of section
have been established by Hryniewicz--Salom\~ao--Siefring~\cite{hss}.

Even for one of the most simple flows in dimension three, the
Hopf flow on the $3$-sphere, global surfaces of section display
rich features. The paper \cite{agz} determines the topology
of such surfaces of section and describes various ways to construct them.
Furthermore, it was shown there how to relate such surfaces of section
to algebraic curves in the complex projective plane, and how to interpret
them as a toy model for the elementary degenerations arising in
symplectic field theory; this led to a new proof of the degree-genus
formula for complex projective curves.

The aim of the present paper is to extend the results of \cite{agz}
to all flows on $3$-manifolds that define a Seifert fibration.
We decide the existence question and give a
topological classification of the global surfaces of section
for such flows. For some classification statements we
restrict attention to \emph{positive}
surfaces of section, which means that the boundary components are supposed
to be oriented positively by the flow (see Section~\ref{section:sos}
for details). This is motivated by the particular interest of these
positive sections, for instance in Reeb dynamics.

We discuss branched coverings between surfaces of section for the
Hopf flow and those of any other Seifert fibration of the
$3$-sphere~$S^3$. Rather intriguingly, these branched coverings can go
either way, and both approaches can be used to compute the genus of
positive $d$-sections with a Riemann--Hurwitz argument.

We also relate the surfaces of section in $S^3$ to algebraic curves
in \emph{weighted} complex projective planes. On the one hand,
the degree-genus formula for such algebraic curves
can then be used to compute, once again, the genus of positive $d$-sections.
Conversely, our direct geometric arguments for surfaces of section
provide an alternative proof of this degree-genus formula.
These considerations have contributed to
filling a gap in the earlier versions of~\cite{hosg20},
and they give a more geometric alternative to the arguments
in~\cite{orwa72}.

We hope that the explicit descriptions of global surfaces of section
in the present paper will prove useful for the study of Besse manifolds
(cf.\ Section~\ref{subsection:Besse}) in Finsler geometry and Reeb dynamics.
\section{Seifert manifolds}
\label{section:Seifert}
Let $M$ be a closed, oriented $3$-manifold with a Seifert fibration
$M\rightarrow B$ over a closed, oriented surface~$B$.
Many of the results in this paper also hold,
\emph{mutatis mutandis}, when $M$ or $B$ is not orientable.
However, for the interpretation of the Seifert fibres as the
orbits of a flow, and to be able to speak of \emph{positive}
surfaces of section, the fibres have to be oriented.

We refer the reader to \cite{jane83} for the basic theory of
Seifert fibred spaces.
\subsection{Seifert invariants}
We recall the notation and
conventions for the Seifert invariants from~\cite{gela18}
that we shall use in the present paper. In the description of $M$ as
\[ M=M\bigl(g;(\alpha_1,\beta_1),\ldots, (\alpha_n,\beta_n)\bigr),\]
the non-negative integer $g$ stands for the genus of~$B$,
the $\alpha_i$ are positive integers giving the multiplicities of the
singular fibres, and each $\beta_i$ is an integer coprime with the
respective~$\alpha_i$, describing the local behaviour near
the singular fibre. A pair $(\alpha_i,\beta_i)$ with $\alpha_i=1$
does not actually give rise to a singular fibre,
but it corresponds to a modification of the fibration
that contributes to the Euler number.

The non-singular fibres are called regular. Near any regular fibre,
the Seifert fibration is a trivial $S^1$-bundle.

When we speak of a `Seifert manifold $M$', we always mean that
a Seifert fibration has been chosen on the $3$-manifold~$M$.
Some $3$-manifolds admit non-isomorphic Seifert fibrations~\cite{gela18,gela}.

The topological interpretation of these invariants is as follows.
Let $B_0$ be the surface with boundary obtained from
$B$ by removing the interior of $n$ disjoint discs
$D^2_1,\ldots,D^2_n\subset B$.
Denote by $M_0=B_0\times S^1\rightarrow B_0$ the trivial
$S^1$-bundle over~$B_0$.
Write $S^1_1,\ldots, S^1_n$ for the boundary circles of $B_0$,
with the \emph{opposite} of the orientation induced as
boundary of~$B_0$.
We write the (isotopy class of) the fibre in $M_0$ as
$h=\{*\}\times S^1$. On the boundary of $M_0$ we consider the
curves $q_i=S^1_i\times\{1\}$, with $1\in S^1\subset\C$.
Let $V_i=D^2\times S^1$, $i=1,\ldots,n$, be $n$ copies of a solid torus
with respective meridian and longitude
\[ \mu_i=\partial D^2\times\{1\},\;\; \lambda_i=\{1\}\times S^1\subset
\partial V_i.\]
The Seifert invariants $(\alpha_i,\beta_i)$ then encode the
identifications
\begin{equation}
\label{eqn:ident1}
\mu_i=\alpha_i q_i+\beta_i h,\;\;\; \lambda_i=\alpha_i'q_i+\beta_i'h,
\end{equation}
where integers $\alpha_i',\beta_i'$ are chosen such that
\[ \begin{vmatrix}
\alpha_i & \alpha_i'\\
\beta_i  & \beta_i'
\end{vmatrix}=1.\]
The identifications \eqref{eqn:ident1} are equivalent to
\begin{equation}
\label{eqn:ident2}
h=-\alpha_i'\mu_i+\alpha_i\lambda_i,\;\;\; q_i=\beta_i'\mu_i-\beta_i\lambda_i.
\end{equation}
Writing $C_i$ for the spine $\{0\}\times S^1$ of the solid torus~$V_i$,
we deduce from \eqref{eqn:ident2}
that the fibre $h$ is homotopic to $\alpha_iC_i$ in~$V_i$,
and the boundary curve~$q_i$, to $-\beta_iC_i$. For $\alpha_i>1$,
the spine $C_i$ is the corresponding singular fibre.

The \emph{Euler number} of the Seifert fibration is defined as
\[ e=-\sum_{i=1}^n\frac{\beta_i}{\alpha_i}.\]
\subsection{Equivalences of Seifert invariants}
\label{subsection:equivalences}
Up to isomorphism, a Seifert fibration is determined by the
Seifert invariants. Different sets of Seifert invariants
correspond to isomorphic Seifert fibrations if and only if they
are related to each other by the following operations,
see \cite[Theorem~1.5]{jane83}:
\begin{itemize}
\item[(o)] Permute the $n$ pairs $(\alpha_i,\beta_i)$.
\item[(i)] Add or delete any pair $(\alpha,\beta)=(1,0)$.
\item[(ii)] Replace each $(\alpha_i,\beta_i)$ by
$(\alpha_i,\beta_i+k_i\alpha_i)$, where $\sum_{i=1}^n k_i=0$.
\end{itemize}
These operations allow one to write a Seifert
manifold in terms of \emph{normalised} Seifert invariants
\[ M\bigl(g;(1,b),(\alpha_1,\beta_1),\ldots,(\alpha_n,\beta_n)\bigr),\]
where $b\in\Z$ and $0<\beta_i<\alpha_i$. Notice that the Euler number
is invariant under these operations.
In general, we shall not be using normalised invariants.
Occasionally it is convenient to collect all pairs $(\alpha_i,\beta_i)$
with $\alpha_i=1$ into a single pair $(1,b)$, so that the remaining
pairs actually correspond to singular fibres.
\subsection{Surgery description}
For Seifert manifolds with $g=0$, one can easily translate
the gluing prescriptions into a surgery diagram, and the equivalences
of Seifert invariants then correspond to well-known transformations
of surgery diagrams.

A surgery description of $S^2\times S^1$ is given by a $0$-framed unknot
$K_0$ in the $3$-sphere~$S^3$.
Any meridian $m_i$ of $K_0$ corresponds to an $S^1$-fibre
of $S^2\times S^1\rightarrow S^2$. The curves $q_i$ then correspond
to a meridian of $m_i$, and the fibre $h$ defines the Seifert framing
of $m_i$ in the unsurgered~$S^3$. So the gluing instruction
$\mu_i=\alpha_iq_i+\beta_ih$ from \eqref{eqn:ident1}
amounts to a surgery with coefficient $\alpha_i/\beta_i$ along~$m_i$.
It follows that the surgery diagram
for $M(0;(\alpha_1,\beta_1),\ldots,(\alpha_n,\beta_n))$
consists of a $0$-framed unknot $K_0$ and $n$ meridians with surgery framing
$\alpha_i/\beta_i$, $i=1,\ldots,n$.

Equivalence (i) in Section~\ref{subsection:equivalences}
corresponds to adding or removing an $\infty$-framed meridian of $K_0$,
which defines a trivial surgery. Equivalence (ii) comes from
performing Rolfsen twists (cf.~\cite[Section 9.H]{rolf76})
on the $n$ meridians. Surgery along $m_i$ is defined by removing
a tubular neighbourhood of $m_i$, and then regluing
it according to the surgery instruction given by the surgery coefficient.
A Rolfsen twist amounts to performing a $k_i$-fold Dehn twist,
$k_i\in\Z$, along the `neck' created by removing the tubular neighbourhood
of~$m_i$ before regluing a solid torus.
This changes the surgery framing of $K_0$ by~$k_i$, and
the surgery coefficient of $m_i$ to $\alpha_i/(\beta_i+k_i\alpha_i)$,
since the Dehn twist adds $k_i$ meridians to a parallel of $K_0$,
and it sends $\alpha_iq_i+\beta_ih$ to
$\alpha_iq_i+(\beta_i+k_i\alpha_i)h$. Observe that these Dehn twists
preserve the $S^1$-fibration. The condition
$\sum_i k_i=0$ comes from the requirement that
the unknot $K_0$ should retain its $0$-framing after the $n$ Rolfsen
twists, so that surgery along $K_0$ still realises $S^2\times S^1$.

For more on surgery descriptions of Seifert manifolds, see~\cite{save99}.
\subsection{Besse flows}
\label{subsection:Besse}
Flows of non-singular vector fields where all orbits are closed
are often referred to as \emph{Besse flows}, in particular in
Reeb dynamics; see~\cite{ggm,kela,mara}. In this context it is worth
mentioning that, by a classical result of Epstein~\cite{epst72},
any Besse flow on a $3$-manifold defines a Seifert fibration.

\section{Global surfaces of section}
\label{section:sos}
The orientations of $M$ and $B$ in the Seifert fibration
$M\rightarrow B$ induce an orientation on the fibres.
Equivalently, we may think of the Seifert fibres as the orbits
of an effective, fixed point free $S^1$-action. A singular fibre
with Seifert invariants $(\alpha,\beta)$, $\alpha>1$, consists
of points with isotropy group
$\Z_{\alpha}=\bigl\langle\rme^{2\pi\rmi/\alpha}\bigr\rangle<S^1$.

A \textbf{global surface of section}
is an embedded compact surface $\Sigma\subset M$ whose boundary
is a collection of Seifert fibres, and whose interior intersects
all other Seifert fibres transversely. The interior of $\Sigma$ will
intersect every regular fibre (except those forming the
boundary of~$\Sigma$) in the same number $d$ of points.
A singular fibre of multiplicity $\alpha$ can also arise as
a boundary component, or it will intersect the interior of $\Sigma$
in $d/\alpha$ points. For short, we speak of a \textbf{$d$-section}.

\begin{note}
\label{note:divisibility}
Given a $d$-section $\Sigma$, the multiplicity $\alpha$ of a singular fibre
not in the boundary of $\Sigma$ divides~$d$.
\end{note}

We orient $\Sigma$ such that these intersection points are positive.
The boundary orientation of a component of $\partial\Sigma$
may or may not coincide with the orientation as a Seifert fibre.
If the orientations coincide for every component of $\partial\Sigma$,
we call the $d$-section \emph{positive}.

See \cite{agz} for a variety of explicit surfaces of section for
the Hopf fibration $S^3\rightarrow S^2$.
\section{$1$-sections for Seifert fibrations}
We begin with a necessary criterion for the existence of a $1$-section.

\begin{lem}
\label{lem:1-section}
If a Seifert manifold admits a $1$-section, then every singular fibre
is a boundary fibre, and the corresponding pair $(\alpha,\beta)$
satisfies $\beta\equiv\pm 1$ $\mathrm{mod}~\alpha$, depending on whether
the singular fibre is a positive or negative boundary component.
\end{lem}

\begin{proof}
The first claim follows from Note~\ref{note:divisibility}.
For the second statement, we consider a solid torus $V$
around the singular fibre $C$ as in Section~\ref{section:Seifert}.
Let $\sigma$ be the curve $\partial V\cap\Sigma$, oriented as a
boundary component of $\Sigma\setminus\Int(V)$. The curve $\sigma$ must be
homotopic in $V$ to~$\pm C$, depending on the sign of the boundary
fibre~$C$.

The Seifert fibres lying in $\partial V$ represent the
fibre class~$h$. Since these fibres intersect $\Sigma$ positively
in a single point, the intersection number $\sigma\bullet h$
on $\partial V$, computed with respect to the oriented basis $(\mu,\lambda)$,
equals~$-1$ (\emph{sic}!). Since $q\bullet h=1$, this means that
$\sigma$ can be written as $\sigma=-q+ah$ for some $a\in\Z$, which
in $V$ is homotopic to $(\beta+a\alpha)C$. For this to equal $\pm C$,
we must have $\beta\equiv \pm 1$ mod~$\alpha$.
\end{proof}

This lemma implies that a Seifert manifold admitting a $1$-section is of
the form $M\bigl(g;(1,b),(\alpha_1,\pm 1),\ldots,(\alpha_n,\pm 1)\bigr)$,
where we may assume that the $\alpha_i$ are greater than~$1$. The individual
signs determine the topology of the Seifert manifold. For $\alpha_i=2$,
both signs are possible.

We now show that manifolds of this form \emph{do} admit a $1$-section,
and we can characterise those admitting a positive one.

\begin{thm}
\label{thm:1-section}
A Seifert manifold admits a $1$-section with $n$ singular boundary
fibres and $b_{\pm}$ positive resp.\ negative regular boundary fibres
if and only if it is isomorphic to
$M\bigl(g;(1,b_+ - b_-),(\alpha_1,\pm 1),\ldots,(\alpha_n,\pm 1)\bigr)$.
The sign in $(\alpha_i,\pm 1$) determines the sign
of the corresponding singular boundary fibre.
Any two $1$-sections on a given Seifert manifold are isotopic,
provided they have the same number of positive resp.\ negative
regular boundary fibres.
\end{thm}

\begin{ex}
In \cite[Section~2.5]{agz} the reader finds an explicit example
of a pair of pants $1$-section for the Hopf fibration
with two positive and one negative boundary fibres, corresponding to
the description of the Hopf fibration as the Seifert manifold
$M\bigl(0;(1,1),(1,1),(1,-1)\bigr)$.
\end{ex}

\begin{proof}[Proof of Theorem~\ref{thm:1-section}]
Given a Seifert manifold of the described form,
we construct a $1$-section as follows. Remove (open) solid tori
around the $n$ singular fibres, and a further $b_+ + b_-$ solid tori
around regular fibres. The remaining part $M_0$ of $M$
is of the form $B_0\times S^1\rightarrow B_0$,
and we take the constant section
$B_0\times\{1\}$ as the part of our desired $1$-section over~$B_0$.

For the gluing of a solid torus $V$ corresponding to a singular fibre
of type $(\alpha,\pm 1)$, we may take
$\alpha'=\mp 1$ and $\beta'=0$. This means we have the
identifications
\[ h=\pm\mu+\alpha\lambda,\;\;\; q=\mp\lambda.\]
Therefore, the oriented vertical annulus $A$ in $V$ with boundary
$\partial A=\pm C\mp\lambda=\pm C+q$ will intersect
each Seifert fibre in $V\setminus\{C\}$ transversely in a
single positive point, and it glues with $B_0$
to extend the $1$-section. (For the sign of the intersection point,
notice that $-q=\pm \lambda$ takes the role of $\sigma$
in the preceding proof, and $\pm\lambda\bullet h=-1$.)
Observe that this argument did not require $\alpha>1$. Thus,
in the same way we may extend the $1$-section over $b_{\pm}$
solid tori glued according to the prescription $(1,\pm 1)$.

Conversely, given a $1$-section as described, we can use
it to trivialise the Seifert fibration outside solid tori
around the boundary fibres. Then $a=0$ in the notation
of the proof of Lemma~\ref{lem:1-section}, and $\beta=\pm 1$.

The uniqueness statement is proved just like \cite[Proposition~3.1]{agz}.
The only difference is that we now have singular fibres,
but if we fix the number of positive resp.\ negative regular boundary fibres,
the signs in $(\alpha_i,\pm 1)$ must be the same for any two given sections,
even when some of the $\alpha_i$ equal~$2$ (perhaps after reordering).
This implies that the two sections can be assumed to coincide near
the singular fibres after a preliminary isotopy.
\end{proof}
\section{$d$-sections for Seifert fibrations}
\label{section:d}
Notice that in Theorem~\ref{thm:1-section} it was not necessary to
assume $\alpha_i>1$. We continue to work with unnormalised
Seifert invariants, although occasionally it may
be convenient to collect all pairs $(\alpha_i,\beta_i)$
with $\alpha_i=1$ into a single term.
\subsection{A necessary criterion}
Again we start with a necessary criterion for the existence of
a section.

\begin{lem}
\label{lem:d-section}
Suppose $M\bigl(g;(\alpha_1,\beta_1),\ldots,(\alpha_n,\beta_n)\bigr)$
admits a $d$-section. If the singular fibre $C=C_i$ with
invariants $(\alpha,\beta)=(\alpha_i,\beta_i)$, $\alpha_i>1$,
is not a boundary fibre,
then $\alpha|d$. If $C$ is a boundary fibre, then
$\alpha|(d\beta\mp 1)$, where the sign depends on $C$ being a positive
or negative boundary.
\end{lem}

\begin{rem}
For $\alpha>1$ the two divisibility conditions are mutually exclusive
(in the second case,
$\alpha$ and $d$ are coprime), so it is
determined \emph{a priori} --- provided a $d$-section exists ---
which singular fibres will be boundary fibres.
Also, the sign of the boundary fibres
is predetermined, unless $\alpha=2$.
\end{rem}

\begin{proof}[Proof of Lemma~\ref{lem:d-section}]
For $C$ not being a boundary fibre, use
Note~\ref{note:divisibility}. If $C$ is a boundary fibre, we argue
as in the proof of Lemma~\ref{lem:1-section}, except that now we
have $\sigma\bullet h=-d$. This implies that $\sigma=-dq+ah$
for some $a\in\Z$, which is homotopic in $V$ to $(d\beta+a\alpha)C$,
whence $d\beta+a\alpha=\pm 1$.
\end{proof}
\subsection{The $\Z_d$-quotient}
If we think of a Seifert fibration as a manifold $M$ with an effective,
fixed point free $S^1$-action, the subgroup $\Z_d=\bigl\langle
\rme^{2\pi\rmi/d}\bigr\rangle<S^1$ also acts on~$M$,
and the quotient $M/\Z_d$ is again Seifert fibred. The
regular fibres of $M$ have length $2\pi$, a singular fibre of
multiplicity $\alpha$ has length $2\pi/\alpha$. In the quotient
$M/\Z_d$, the regular fibres have length $2\pi/d$, and a singular
fibre of multiplicity $\alpha$ in $M$ descends to a singular
fibre of length $2\pi\gcd(\alpha,d)/\alpha d$, so its multiplicity
is $\alpha/\gcd(\alpha,d)$.

The Seifert invariants of $M/\Z_d$ can be computed from
a section of $M\rightarrow B$ over $B_0\subset B$ (in the notation
of Section~\ref{section:Seifert}), since this descends to
just such a section of $M/\Z_d\rightarrow B$,
see \cite[Proposition~2.5]{jane83}.

\begin{prop}
\label{prop:Zd-quotient}
The $\Z_d$-quotient $\oM=M/\Z_d$ of
\[ M=M\bigl(g;(\alpha_1,\beta_1),\ldots,(\alpha_n,\beta_n)\bigr)\]
is the Seifert manifold
\[ \oM=M\bigl(g;(\oalpha_1,\obeta_1),\ldots,(\oalpha_n,\obeta_n)
\bigr),\]
where $\oalpha_i=\alpha_i/\gcd(\alpha_i,d)$ and
$\obeta_i=d\beta_i/\gcd(\alpha_i,d)$.
In particular, the Euler numbers are related by
$e(\oM\rightarrow B)=d\cdot e(M\rightarrow B)$.\qed
\end{prop}

The quotient map $M\rightarrow\oM$ is a branched covering,
branched transversely to any singular fibre with
$\gcd(\alpha_i,d)>1$,
where the branching index is precisely this greatest common divisor.
\subsection{$d$-sections descend to the $\Z_d$-quotient}
\label{subsection:descend}
We now assume that $M$ admits a $d$-section, so that the conclusion
of Lemma~\ref{lem:d-section} holds. We permute the Seifert invariants
such that (for the appropriate $k\in\{0,\ldots,n\}$)
\begin{equation}
\label{eqn:alternative}
\begin{cases}
\alpha_i|(d\beta_i\mp 1)\;\text{and}\;\alpha_i\nmid d
        & \text{for $i=1,\ldots,k$},\\
\alpha_i|d
        & \text{for $i=k+1,\ldots,n$}.
\end{cases}
\end{equation}
The assumption $\alpha_i\nmid d$ in the first case of
\eqref{eqn:alternative} is not, strictly speaking, necessary.
It avoids the ambiguity where to place the $\alpha_i$ that equal~$1$,
but the formulas we shall derive also hold when we count some
of the $\alpha_i=1$ amongst the first alternative.
For $i=1,\ldots, k$, we write $d\beta_i=a_i\alpha_i+\varepsilon_i$
with $a_i\in\Z$ and $\varepsilon_i\in\{\pm 1\}$. Then
\[ \oM=M\bigl(g;(1,\ob),(\alpha_1,\varepsilon_1),\ldots,
(\alpha_k,\varepsilon_k)\bigr) \]
with
\[ \ob=\sum_{i=1}^k a_i+\sum_{i=k+1}^n \frac{d\beta_i}{\alpha_i}.\]
For $\alpha_i=2$ and $d$ odd, or for $\alpha_i=1$ counted amongst
the first alternative, there is an ambiguity in the choice of
$\varepsilon_i$ (and the corresponding~$a_i$), but this does
not invalidate any of our statements.

With Theorem~\ref{thm:1-section} we see that
$\oM$ always admits a $1$-section, and a positive one
precisely when $\ob\geq 0$ and $\varepsilon_i=1$ for $i=1,\ldots,k$.
The Euler number $\ole$ of $\oM\rightarrow B$ equals
\[ \ole=-\ob-\sum_{i=1}^k\varepsilon_i/\alpha_i, \]
so the condition
for $\oM$ to admit a positive $1$-section can be rephrased as:
all $\varepsilon_i$ equal $1$ and $\ole\leq-\sum_{i=1}^k1/\alpha_i$.

In fact, by the same topological arguments as
in \cite[Proposition~3.2]{agz} one can show that any $d$-section of
$M$ can be isotoped to one that is invariant under the $\Z_d$-action,
and it then descends to a $1$-section of~$\oM$.
Conversely, any $1$-section of $\oM$
lifts to a $\Z_d$-invariant $d$-section of~$M$. Also,
a $d$-section of $M$, if it exists,
is determined up to isotopy by its boundary orientations.
For simplicity, we shall restrict the discussion of the
topology of these $d$-sections to positive ones.
\subsection{The local behaviour near singular fibres}
To gain a better understanding of $d$-sections, we briefly describe
their behaviour near singular fibres. We write $(\alpha,\beta)$
for the invariants of the fibre $C$ in question.

If $\alpha|(d\beta\pm 1)$ and $\alpha\nmid d$, then
$C$ is a boundary component of the $d$-section. Near $C$,
the $d$-section looks like a helicoidal surface, where
the number of turns around $C$ is controlled by the condition
that any neighbouring regular fibre cuts the surface in $d$ (positive)
points.

If $\alpha|d$, then the $d$-section near $C$ consists of
$d/\alpha$ transverse discs to~$C$.

The quotient map $M\rightarrow\oM$ restricts on any $\Z_d$-invariant
$d$-section $\Sigma$ of $M$ to a branched covering $\Sigma\rightarrow
\oSigma$, with $\oSigma$ a $1$-section of~$\oM$. Each
singular fibre with $\alpha_i|d$ gives rise to $d/\alpha_i$
branch points of index $\alpha_i$ over a single point in the
interior of~$\oSigma$. Along the boundary components,
regular or singular, the covering $\Sigma\rightarrow\oSigma$
is unbranched.
\subsection{The topology of positive $d$-sections}
Here is our main theorem concerning positive $d$-sections.

\begin{thm}
\label{thm:d-section}
The Seifert manifold $M=M\bigl(g;(\alpha_1,\beta_1),\ldots,
(\alpha_n,\beta_n)\bigr)$ admits a positive $d$-section if and only
if, for some $k\in\{0,\ldots,n\}$ (and after permuting the
Seifert invariants), condition \eqref{eqn:alternative}
is satisfied with the sign $\alpha_i|(d\beta_i-1)$ in
the first alternative, and the Euler number $e$ satisfies
\begin{equation}
\label{eqn:ob}
de+\sum_{\alpha_i\nmid d}\frac{1}{\alpha_i}\leq 0.
\end{equation}
This positive $d$-section, if it exists, is unique up to isotopy.
Unless $\alpha_i|d$ for all $i$ and $e=0$, it is a connected, orientable
surface of genus
\[ \frac{1}{2}\left(2-
d\Bigl(2-2g+(d-1)e+\sum_{i=1}^n\frac{1}{\alpha_i}-n\Bigr)
-\sum_{\alpha_i\nmid d}\Bigl(1-\frac{1}{\alpha_i}\Bigr) \right)\]
with
\[ -de+\sum_{\alpha_i\nmid d}\Bigl(1-\frac{1}{\alpha_i}\Bigr) \]
boundary components.
\end{thm}

\begin{proof}
The left-hand side of inequality \eqref{eqn:ob}
equals the integer $-\ob$, so the inequality is simply the reformulation
of the condition $\ob\geq 0$. The sign condition on the first
alternative of \eqref{eqn:alternative} is equivalent to
requiring $\varepsilon_1=\ldots=\varepsilon_k=1$.

The number of boundary components of the positive $d$-section $\Sigma$ equals
that of the $1$-section $\oSigma$ of~$\oM$, which by
Theorem~\ref{thm:1-section} is
\[ \ob+k=-de+\sum_{\alpha_i\nmid d}\Bigl(1-\frac{1}{\alpha_i}\Bigr).\]
If $\Sigma$ has empty boundary, which happens when $e=0$ and $\alpha_i|d$
for all $i=1,\ldots,n$, it may be disconnected, depending on
the branching of the map $\Sigma\rightarrow\oSigma$; an obvious case
is a trivial $S^1$-bundle $M=B\times S^1\rightarrow B$. When the
boundary is non-empty, $\Sigma$ is connected by the argument in
the proof of \cite[Theorem~4.1]{agz}. For the remainder of the proof
we assume this latter case.

Write $\hat{\Sigma}$ for the closed surface obtained by gluing $\ob+k$
copies of $D^2$ to the $d$-section $\Sigma$ of~$M$ along its boundary
components, and $\hat{\oSigma}$ for the closed surface obtained in the same
way from the $1$-section $\oSigma$ of~$\oM$.
This gluing is done abstractly, that is, the surfaces
$\hat{\Sigma}$ and $\hat{\oSigma}$ are no longer embedded surfaces
in $M$ or $\oM$, respectively. The Euler characteristic
of $\hat{\oSigma}$ equals $2-2g$, since the interior of the $1$-section
is a copy of the base $B$ with $\ob+k$ points removed.

The surface $\hSigma$ is a branched cover of $\hat{\oSigma}$ as
follows. Each of the discs glued to one
of the $\ob+k$ boundary components gives rise
to a single branch point in $\hSigma$ of index~$d$, since both
a regular fibre in $M$ as well as a singular fibre with $\alpha_i$
coprime to $d$ is a $d$-fold cover of the corresponding fibre in~$\oM$.

Each singular fibre of $M$ with $\alpha_i|d$
gives rise to a point in $\hat{\oSigma}$ covered by $d/\alpha_i$ points
of branching index~$\alpha_i$. This statement
about covering points and indices remains true for $\alpha_i=1$.

Thus, the Riemann--Hurwitz formula for the branched covering $\hSigma
\rightarrow\hat{\oSigma}$ gives
\[ \chi(\hSigma)=d\bigl(2-2g-(\ob+k)-(n-k)\bigr)+(\ob+k)+
\sum_{i=k+1}^n\frac{d}{\alpha_i}.\]
The formula for the genus $(2-\chi(\hSigma))/2$ of the
$d$-section follows with some simple arithmetic. Notice that the
formula is invariant under adding or removing pairs $(1,0)$
from the Seifert invariants.
\end{proof}
\subsection{Examples}
Here are some elementary applications of Theorem~\ref{thm:d-section}
\subsubsection{The Hopf flow on $S^3$}
The Hopf fibration $S^3\rightarrow S^2$ has the description
$M\bigl(0;(1,1)\bigr)$ as a Seifert manifold, with $e=-1$, see
\cite[Section~2.4]{agz}. The following is \cite[Theorem~4.1]{agz}.

\begin{ex}
For any natural number $d$, the Hopf fibration admits
a unique positive $d$-section, which has genus $(d-1)(d-2)/2$
and $d$ boundary components.
\end{ex}
\subsubsection{Seifert fibrations of $S^3$}
\label{subsubsection:S3}
According to \cite[Proposition~5.2]{gela18}, a complete list
of the Seifert fibrations of $S^3$ is provided by
\[ M\bigl(0;(\alpha_1,\beta_1),(\alpha_2,\beta_2)\bigr),\]
where $\alpha_1\beta_2+\alpha_2\beta_1=1$ (in particular,
$\alpha_1$ and $\alpha_2$ are coprime). In order to avoid
duplications in this list, one may assume $\alpha_1\geq\alpha_2$
and $0\leq\beta_1<\alpha_1$. For our purposes, this is not relevant,
and so we shall not insist on this choice, since we prefer to
retain the symmetry in $\alpha_1$ and~$\alpha_2$.

The corresponding $S^1$-action on $S^3\subset\C^2$ is given by
\begin{equation}
\label{eqn:action21}
\tag{$A_{2,1}$}
\theta(z_1,z_2)=
\bigl(\rme^{\rmi\alpha_2\theta}z_1,\rme^{\rmi\alpha_1\theta}z_2\bigr),
\;\;\;\theta\in\R/2\pi\Z.
\end{equation}
Notice the choice of indices. The fibre $C_1=\{z_2=0\}\cap S^3$ through the
point $(1,0)$ has multiplicity~$\alpha_2$, the fibre
$C_2=\{z_1=0\}\cap S^3$ through $(0,1)$,
multiplicity~$\alpha_1$.
The Euler number of this fibration is
\[ e=-\frac{\beta_1}{\alpha_1}-\frac{\beta_2}{\alpha_2}=
-\frac{1}{\alpha_1\alpha_2}.\]
The Hopf fibration corresponds to the choice $\alpha_1=\alpha_2=1$.

We shall discuss these Seifert fibrations in more detail
in Section~\ref{section:S3}. The following example answers a question
posed to us by Christian Lange.

\begin{ex}
The Seifert fibration $M\bigl(0;(\alpha_1,\beta_1),(\alpha_2,\beta_2)\bigr)$
of $S^3$ admits a positive $\alpha_1\alpha_2$-section, with boundary a single
regular fibre, of genus $(\alpha_1-1)(\alpha_2-1)/2$.
\end{ex}

Here is one way how to describe this $\alpha_1\alpha_2$-section explicitly.
A more systematic treatment, using algebraic curves in
weighted projective spaces, will be given in Sections~\ref{section:S3}
and~\ref{section:algebraic}.
We replace the round $3$-sphere $S^3\subset\C^2$ by its diffeomorphic copy
\[ S^3_{\alpha_1,\alpha_2}:=
\bigl\{(z_1,z_2)\in\C^2\co|z_1|^{2\alpha_1}+|z_2|^{2\alpha_2}=1\bigr\},\]
The $S^1$-action is as in~\eqref{eqn:action21}, and we write
$C_1^{\alpha_1,\alpha_2}$, $C_2^{\alpha_1,\alpha_2}$ for the fibres
through the points $(1,0)$ and $(0,1)$, respectively.
Now consider the map
\begin{equation}
\label{eqn:rho}
\begin{array}{rccc}
\rho_{\alpha_1,\alpha_2}\co & S^3_{\alpha_1,\alpha_2} & \longrightarrow
    & S^3\\
                            & (z_1,z_2)               & \longmapsto
    & (z_1^{\alpha_1},z_2^{\alpha_2}).
\end{array}
\end{equation}
We have
\[ \rho_{\alpha_1,\alpha_2}\bigl(\rme^{\rmi\alpha_2\theta}z_1,
\rme^{\rmi\alpha_1\theta}z_2\bigr)=\bigl(\rme^{\rmi\alpha_1\alpha_2\theta}
z_1^{\alpha_1},\rme^{\rmi\alpha_1\alpha_2\theta}z_2^{\alpha_2}\bigr),\]
so each $S^1$-orbit on $S^3_{\alpha_1,\alpha_2}$ maps $\alpha_1\alpha_2:1$
to a Hopf orbit on $S^3$, except for the two exceptional orbits
$C_i^{\alpha_1,\alpha_2}$, which are $\alpha_i$-fold
coverings of the Hopf orbits $C_i$, $i=1,2$. So
$\rho_{\alpha_1,\alpha_2}$ is a covering branched along
(and transverse to) $C_1^{\alpha_1,\alpha_2}$, $C_2^{\alpha_1,\alpha_2}$
with branching index $\alpha_2,\alpha_1$, respectively.

\begin{note}
Points with the same image under $\rho_{\alpha_1,\alpha_2}$
do indeed come from the same $S^1$-orbit on $S^3_{\alpha_1,\alpha_2}$.
For instance, the point $\bigl(\rme^{2\pi\rmi/\alpha_1}z_1,z_2\bigr)$
lies on the same orbit as $(z_1,z_2)$, since there is a $k\in\Z$
with $k\alpha_2\equiv 1$ mod~$\alpha_1$, and then with
$\theta:=2\pi k/\alpha_1$ we have $\theta(z_1,z_2)=
\bigl(\rme^{2\pi\rmi/\alpha_1}z_1,z_2\bigr)$.
\end{note}

It follows that the quotient map $\pi_{2,1}\co S^3_{\alpha_1,\alpha_2}
\rightarrow S^3_{\alpha_1,\alpha_2}/S^1=\CP^1$ is given by
\[ \pi_{2,1}(z_1,z_2)=[z_1^{\alpha_1}:z_2^{\alpha_2}]\in\CP^1=S^2.\]
Observe that the preimage of a point $[w_1:w_2]\in\CP^1$, where
we may assume that $|w_1|^2+|w_2|^2=1$,
consists of the points $\bigl(\lambda_1(w_1)^{1/\alpha_1},
\lambda_2(w_2)^{1/\alpha_2}\bigr)\in S^3_{\alpha_1,\alpha_2}$
for some fixed choice
of roots and $\lambda_1,\lambda_2\in S^1\subset\C$ with
$\lambda_1^{\alpha_1}=\lambda_2^{\alpha_2}$, which means
we can write $\lambda_1=\rme^{\rmi\alpha_2\theta}$ and $\lambda_2=
\rme^{\rmi\alpha_1\theta}$ for a suitable $\theta\in\R/2\pi\Z$.

The element $\rme^{2\pi\rmi/\alpha_1\alpha_2}\in S^1$
defines a $\Z_{\alpha_1\alpha_2}$-action on $S^3_{\alpha_1,\alpha_2}$.
In other words, this is the action generated by
\[ (z_1,z_2)\longmapsto \bigl(\rme^{2\pi\rmi/\alpha_1}z_1,
\rme^{2\pi\rmi/\alpha_2}z_2\bigr).\]
By the Chinese remainder theorem, for $k\in\{0,1,\ldots, \alpha_1\alpha_2-1\}$
we obtain the $\alpha_1\alpha_2$ independent choices of roots
of unity $\bigl(\rme^{2\pi\rmi k/\alpha_1},\rme^{2\pi\rmi k/\alpha_2}\bigr)$,
so the map $\rho_{\alpha_1,\alpha_2}$ is in fact the quotient map
under this $\Z_{\alpha_1\alpha_2}$-action.

Summing up, we have the
following commutative diagram:

\begin{diagram}
S^3_{\alpha_1,\alpha_2}  &                                 &   \\
                         & \rdTo^{\rho_{\alpha_1,\alpha_2}} & \\
\dTo^{\pi_{2,1}}         &                                 & S^3=
                      S^3_{\alpha_1,\alpha_2}/\Z_{\alpha_1\alpha_2}\\
                         & \ldTo_{\pi_{\mathrm{Hopf}}}     & \\
\CP^1                    &                                 &
\end{diagram}

\begin{rem}
The quotient $S^3_{\alpha_1,\alpha_2}/\Z_{\alpha_1\alpha_2}$ can also
be recognised as $S^3$ with the Hopf fibration by
computing its Seifert invariants. By Proposition~\ref{prop:Zd-quotient}
we have
\begin{eqnarray*}
S^3_{\alpha_1,\alpha_2}/\Z_{\alpha_1\alpha_2}
& = & M\bigl(0;(1,\alpha_2\beta_1),(1,\alpha_1\beta_2)\bigr)\\
& = & M\bigl(0;(1,0),(1,\alpha_1\beta_2+\alpha_2\beta_1)\bigr)\\
& = & M\bigl(0;(1,0),(1,1)\bigr),
\end{eqnarray*}
which describes the diagonal action of $S^1$ on $S^3\subset\C^2$.
\end{rem}

For the Hopf fibration, it is easy to write down a disc-like
$1$-section, where the boundary fibre sits over the point $[1:1]\in\CP^1$.
Lifting this to $S^3_{\alpha_1,\alpha_2}$, one finds an
$\alpha_1\alpha_2$-section as the image of the $\alpha_1\alpha_2$-valued map
\[ \C\ni z\longmapsto
\left(\Bigl(\frac{z-1}{\sqrt{|z-1|^2+|z|^2}}\Bigr)^{1/\alpha_1},
\Bigl(\frac{z}{\sqrt{|z-1|^2+|z|^2}}\Bigr)^{1/\alpha_2}\right)
\in S^3_{\alpha_1,\alpha_2}.\]
As $z\rightarrow\infty$, this becomes asymptotic to
the (regular) fibre over $[1:1]\in\CP^1$. For $z=0$ the map
is $\alpha_1$-valued. This corresponds to the $\alpha_1$
points over $[1:0]$ of branching index $\alpha_2$. Similarly,
the point $z=1$ maps to the $\alpha_2$ branch points of index $\alpha_1$.
\subsubsection{Existence of $d$-sections}
As a corollary of Theorem~\ref{thm:d-section} and the discussion
in Section~\ref{subsection:descend} we mention the following
existence statement, which is in response to a question by
Umberto Hryniewicz.

\begin{cor}
Every Seifert manifold admits a $d$-section for a suitable
positive integer~$d$. It admits a positive $d$-section
if and only if its Euler number is non-positive.
\end{cor}

\begin{proof}
For the Seifert manifold $M=M\bigl(g;(\alpha_1,\beta_1),\ldots,
(\alpha_n,\beta_n)\bigr)$ we may choose $d$ as a common multiple
of the~$\alpha_i$. Then the $\Z_d$-quotient is of the form
$M\bigl(g;(1,\ob)\bigr)$. This admits a $1$-section, which can be
lifted to a $d$-section of~$M$.

For this choice of~$d$, condition \eqref{eqn:ob} will be satisfied
if and only if $e\leq 0$. If $e$ is positive, \eqref{eqn:ob}
will not be satisfied for any choice of~$d$.
\end{proof}

\begin{rem}
Changing the orientation of $M$ amounts to changing the
sign of the Euler number. So a Seifert manifold always admits
a positive $d$-section for at least one of the two orientations.
If $e\neq 0$, the Seifert fibration can be realised as a Besse
Reeb flow~\cite{kela}.
\end{rem}
\section{Weighted projective planes}
In this section we recall the bare essentials of weighted projective planes
that we shall need to give explicit realisations of surfaces of
section for the Seifert fibrations of~$S^3$. For more details
see~\cite{hosg20}.
\subsection{The definition of weighted projective planes}
Let $a_0,a_1,a_2\in\N$ be a triple of pairwise coprime positive integers.
The quotient of $\C^3\setminus\{(0,0,0)\}$ under the $\C^*$-action given by
\[ \lambda(z_0,z_1,z_2)=(\lambda^{a_0}z_0,\lambda^{a_1}z_1,\lambda^{a_2}z_2)\]
is called the (well-formed) \textbf{weighted complex projective
plane} $\bbP(a_0,a_1,a_2)$ with \textbf{weights} $a_0,a_1,a_2$.
The equivalence class of a point $(z_0,z_1,z_2)$ is written
as $[z_0:z_1:z_2]_{\bfa}$, where $\bfa=(a_0,a_1,a_2)$.

The space $\bbP(a_0,a_1,a_2)$ is a complex manifold of complex dimension~$2$,
except for three cyclic singularities at the points
$[1:0:0]_{\bfa}$, $[0:1:0]_{\bfa}$ and $[0:0:1]_{\bfa}$.
Indeed, the subset
\[ U_0:=\bigl\{[z_0:z_1:z_2]_{\bfa}\co z_0\neq 0\bigr\}\]
can be identified with the quotient $\C^2/\Z_{a_0}$ under the
action of the cyclic group $\Z_{a_0}< S^1\subset\C$ given by
\[ \zeta(w_1,w_2)=(\zeta^{a_1}w_1,\zeta^{a_2}w_2),\;\;\;\zeta\in\Z_{a_0}.\]
If we write $[w_1,w_2]_{a_0}$ for the equivalence class
of $(w_1,w_2)$ in this cyclic quotient, the claimed identification
is given by
\[ \begin{array}{rcl}
U_0                    & \longrightarrow & \C^2/\Z_{a_0}\\
{[z_0:z_1:z_2]}_{\bfa} & \longmapsto     & [z_0^{-a_1/a_0}z_1,
                                          z_0^{-a_2/a_0}z_2]_{a_0}.
\end{array} \]
The ambiguity in the choice of a root of order $a_0$ is accounted
for by the equivalence relation on the right. This map is
well defined, with inverse
\[ [w_1,w_2]_{a_0}\longmapsto [1:w_1:w_2]_{\bfa}. \]
There are analogous descriptions for the subsets $U_i=\{z_i\neq 0\}
\subset\bbP(a_0,a_1,a_2)$ for $i=1,2$.
\subsection{Algebraic curves in weighted projective planes}
\label{subsection:alg-curves}
A complex polynomial $F$ in the variables $z_0,z_1,z_2$
is called a \textbf{degree}-$d$ \textbf{weighted-homogeneous polynomial}
(with a choice of $a_0,a_1,a_2$ understood) if
\[ F(\lambda^{a_0}z_0,\lambda^{a_1}z_1,\lambda^{a_2}z_2)=\lambda^d
F(z_0,z_1,z_2)\]
for some $d\in\N$ and all $(z_0,z_1,z_2)\in\C^3$ and $\lambda\in\C^*$.
Then the zero set $\{F=0\}$ is a well-defined subset of
$\bbP(a_0,a_1,a_2)$.

The polynomial $F$ is called \textbf{non-singular} if
there are \emph{no} solutions
to the system of equations
\[ F=\frac{\partial F}{\partial z_0}=\frac{\partial F}{\partial z_1}=
\frac{\partial F}{\partial z_2}=0\]
in $\C^3\setminus\{(0,0,0)\}$.

\begin{prop}
If $F$ is non-singular, then $\{F=0\}\subset\bbP(a_0,a_1,a_2)$
is a smooth Riemann surface.
\end{prop}

\begin{proof}
The proof is completely analogous to the usual proof for
algebraic curves in $\CP^2=\bbP(1,1,1)$, which can be found
in~\cite{kirw92}, for instance. The only additional issue
is to deal with curves that pass through one of the (at most)
three singular points of $\bbP(a_0,a_1,a_2)$. However, this does not
pose a problem, since these singularities are cyclic.
A holomorphic chart around $0\in\C/\Z_{a_i}$ is given by $z\mapsto z^{a_i}$,
so on the complex curve $\{F=0\}$ these points are non-singular.
\end{proof}

\section{Seifert fibrations of the $3$-sphere}
\label{section:S3}
We now give a systematic and comprehensive description of the
positive $d$-sections for all Seifert fibrations of~$S^3$.
Weighted projective lines appear naturally as
the base of these fibrations. We first give a description of
global surfaces of section using the topological arguments
from Section~\ref{section:d}. A description in terms
of algebraic curves in weighted projective planes will
follow in Section~\ref{section:algebraic}.
\subsection{An alternative relation with the Hopf fibration}
\label{subsection:alternative}
As we shall see present\-ly, it is convenient to replace
the $S^1$-action \eqref{eqn:action21} by
\begin{equation}
\label{eqn:action12}
\tag{$A_{1,2}$}
\theta(z_1,z_2)=
\bigl(\rme^{\rmi\alpha_1\theta}z_1,\rme^{\rmi\alpha_2\theta}z_2\bigr),
\;\;\;\theta\in\R/2\pi\Z.
\end{equation}
As before, $\alpha_1$, $\alpha_2$ may be any pair of coprime natural numbers,
and the $S^1$-actions \eqref{eqn:action12} constitute a complete list
of the Seifert fibrations of the $3$-sphere. In contrast with
Section~\ref{subsubsection:S3}, we now regard these as actions
on the round~$S^3$.

Nonetheless, the spheres $S^3_{\alpha_1,\alpha_2}$
again have to play their part, and as before we consider the map
$\rho_{\alpha_1,\alpha_2}$ defined in~\eqref{eqn:rho}.
We now compute
\[ \rho_{\alpha_1,\alpha_2}\bigl(\rme^{\rmi\theta}z_1,
\rme^{\rmi\theta}z_2\bigr)=\bigl(\rme^{\rmi\alpha_1\theta}z_1^{\alpha_1},
\rme^{\rmi\alpha_2\theta}z_2^{\alpha_2}\bigr),\]
which tells us that each Hopf $S^1$-orbit on $S^3_{\alpha_1,\alpha_2}$
is mapped $1:1$ to an orbit of \eqref{eqn:action12} on~$S^3$,
except for the Hopf orbits $C_i^{\alpha_1,\alpha_2}$,
which are $\alpha_i$-fold coverings of the exceptional
$A_{1,2}$-orbits $C_i$, $i=1,2$. Each regular $A_{1,2}$-orbit
is covered by $\alpha_1\alpha_2$ Hopf orbits.

\begin{rem}
Beware that in the set-up of Section~\ref{subsubsection:S3},
the orbits $C_1^{\alpha_1,\alpha_2}, C_2^{\alpha_1,\alpha_2}$
were of multiplicity $\alpha_2,\alpha_1$, respectively, whereas
here the orbits $C_1,C_2$ are of multiplicity $\alpha_1,\alpha_2$,
respectively.
\end{rem}

Observe that the quotient $S^3/S^1$ under the $A_{1,2}$-action can
be naturally identified with the weighted projective line
$\bbP(\alpha_1,\alpha_2)$, which is topologically a $2$-sphere,
and metrically an orbifold with an $\alpha_1$- and an $\alpha_2$-singularity.
The quotient map $\pi_{1,2}\co S^3\rightarrow S^3/S^1=\bbP(\alpha_1,\alpha_2)$
is given by
\[ \pi_{1,2}(z_1,z_2)=[z_1:z_2]_{(\alpha_1,\alpha_2)}
\in \bbP(\alpha_1,\alpha_2).\]

In conclusion, these maps fit together into the following
commutative diagram, where $\orho_{\alpha_1,\alpha_2}$
sends $[z_1:z_2]\in\CP^1$ to
$[z_1^{\alpha_1}:z_2^{\alpha_2}]_{(\alpha_1,\alpha_2)}\in
\bbP(\alpha_1,\alpha_2)$:
\begin{diagram}
S^3_{\alpha_1,\alpha_2}    & \rTo^{\rho_{\alpha_1,\alpha_2}}
     && S^3=S^3_{\alpha_1,\alpha_2}/\Z_{\alpha_1\alpha_2}\\
\dTo^{\pi_{\mathrm{Hopf}}} &
     && \dTo_{\pi_{1,2}}\\
\CP^1                      & \rTo^{\orho_{\alpha_1,\alpha_2}}
     && \bbP(\alpha_1,\alpha_2)
\end{diagram}

\begin{rem}
Contrast this diagram with the one in Section~\ref{subsubsection:S3}.
Whereas the earlier diagram was used to lift a $1$-section of
$\pi_{\mathrm{Hopf}}$ to an $\alpha_1\alpha_2$-section of~$\pi_{2,1}$,
the diagram here will allow us to lift $d$-sections of $\pi_{1,2}$,
for all admissible~$d$, to $d$-sections of~$\pi_{\mathrm{Hopf}}$.
\end{rem}
\subsection{Positive $d$-sections}
Before we discuss the lifting of $d$-sections, we are
going to derive the classification of the positive $d$-sections for
the Seifert fibrations of $S^3$ as a corollary of Theorem~\ref{thm:d-section}.

\begin{cor}
\label{cor:S3}
Table~\ref{table:S3} constitutes a complete list of the positive
$d$-sections $\Sigma$ of the Seifert
fibrations of $S^3$ defined by~\eqref{eqn:action12}.
\end{cor}

\begin{table}
\def\arraystretch{1.5}
\begin{tabular}{|c|c|c|c|c|}\hline
$d$                                               & \#$(\partial\Sigma)$ 
    & $C_1$                       & $C_2$
 & $\mathrm{genus}(\Sigma)$                                \\ \hline\hline
$k\alpha_1\alpha_2$, $k\in\N$                     & $k$
    & $\not\subset\partial\Sigma$ & $\not\subset\partial\Sigma$
 & $\frac{(k\alpha_1-1)(k\alpha_2-1)+1-k}{2}$ \\[1mm] \hline
$k\alpha_1\alpha_2+\alpha_2$, $k\in\N_0$;          & $k+1$
    & $\subset\partial\Sigma$     & $\not\subset\partial\Sigma$
 & $\frac{(k\alpha_1+1)(k\alpha_2-1)+1-k}{2}$ \\ 
$\alpha_1>1$ &&&&\\ \hline
$k\alpha_1\alpha_2+\alpha_1$, $k\in\N_0$;          & $k+1$
    & $\not\subset\partial\Sigma$ & $\subset\partial\Sigma$
 & $\frac{(k\alpha_1-1)(k\alpha_2+1)+1-k}{2}$ \\ 
$\alpha_2>1$ &&&&\\ \hline
$k\alpha_1\alpha_2+\alpha_1+\alpha_2$, $k\in\N_0$; & $k+2$
    & $\subset\partial\Sigma$     & $\subset\partial\Sigma$
 & $\frac{(k\alpha_1+1)(k\alpha_2+1)-1-k}{2}$ \\
$\alpha_1,\alpha_2>1$ &&&&\\ \hline
\end{tabular}
\vspace{1.5mm}
\caption{Positive $d$-sections $\Sigma$ for the Seifert fibrations of~$S^3$.}
\label{table:S3}
\end{table}

The expression \#$(\partial\Sigma)$ in Table~\ref{table:S3} stands for
the number of boundary components of the $d$-section~$\Sigma$.
In the first line, in case $\alpha_i$ equals~$1$, the regular fibre
$C_i$ may well be a component of $\partial\Sigma$.

\begin{proof}[Proof of Corollary~\ref{cor:S3}]
We need to check under which assumptions on $d$
the necessary and sufficient conditions \eqref{eqn:alternative}
and \eqref{eqn:ob} are satisfied.
Recall that we can write the Seifert fibrations of $S^3$
as $M\bigl(0;(\alpha_1,\beta_1),(\alpha_2,\beta_2)\bigr)$ with
$\alpha_1\beta_2+\alpha_2\beta_1=1$, and the Euler number
of this fibration is $e=-1/\alpha_1\alpha_2$.

If both $\alpha_1$ and $\alpha_2$ divide~$d$, then $d=k\alpha_1\alpha_2$
for some $k\in\N$, and the left-hand side of \eqref{eqn:ob}
equals~$-k$. Notice that one or both of the $\alpha_i$ may
equal~$1$.

Next we consider the case that $\alpha_1|(d\beta_1-1)$ (and
$\alpha_1\nmid d$) and $\alpha_2|d$ (including the case
$\alpha_2=1$). The divisibility condition
$\alpha_1|(d\beta_1-1)$ is equivalent to $\alpha_1$ being a divisor of
\[ d\beta_1-1+\alpha_1\beta_2=d\beta_1-\alpha_2\beta_1
=\beta_1(d-\alpha_2),\]
and hence to $\alpha_1|(d-\alpha_2)$. This means that $d$ has to be
of the form $d=k\alpha_1\alpha_2+\alpha_2$ for some
$k\in\N_0=\{0,1,2,\ldots\}$. The case with the roles
of $\alpha_1$ and $\alpha_2$ exchanged is analogous.

If $\alpha_i|(d\beta_i-1)$, $i=1,2$, then the same argument
as before shows that $\alpha_1|(d-\alpha_2)$ and $\alpha_2|(d-\alpha_1)$.
For $\alpha_i\nmid d$ (and in particular $\alpha_i>1$) these
divisibility conditions are equivalent to
$d=k\alpha_1\alpha_2+\alpha_1+\alpha_2$.

The remaining data in Table~\ref{table:S3} on $\#(\partial\Sigma)$
and the genus of $\Sigma$
follow from Lemma~\ref{lem:d-section} and Theorem~\ref{thm:d-section}.
\end{proof}

\begin{rem}
Notice that every Seifert fibration of $S^3$ admits a disc-like
$d$-section, that is, a section of genus zero with one boundary component.
This underlines the subtlety of the construction in~\cite{koer20}
of a Reeb flow on $S^3$ without any disc-like global surface
of section; see also~\cite{kkk}.
\end{rem}
\subsection{Lifting $d$-sections}
\label{subsection:lifting}
We now want to describe an alternative method for computing the genus
of $d$-sections, based on lifting a $d$-section of $\pi_{1,2}$
to a $d$-section of $\pi_{\mathrm{Hopf}}$ in the diagram from
Section~\ref{subsection:alternative}, and then using previous results on
surfaces of section for the Hopf flow~\cite{agz}.

Thus, let $\Sigma\subset S^3$ be a positive $d$-section for~$\pi_{1,2}$.
Write $\tSigma\subset S^3_{\alpha_1,\alpha_2}$ for the lift of $\Sigma$
under the $\alpha_1\alpha_2$-fold covering $\rho_{\alpha_1,\alpha_2}\co
S^3_{\alpha_1,\alpha_2}\rightarrow S^3$. As in Table~\ref{table:S3},
we write $k\in\N_0$ for the number of regular boundary components
of~$\Sigma$. Additionally, one or both of $C_1,C_2$ may belong to
the boundary. (If neither of them does, then $k\geq 1$.)

As observed in Section~\ref{subsection:alternative}, each regular fibre
of $\pi_{1,2}$ is covered $1:1$ by $\alpha_1\alpha_2$ Hopf fibres.
This contributes $k\alpha_1\alpha_2$ boundary components to~$\tSigma$.
Away from the boundary, $\tSigma$ is a $d$-section for $\pi_{\mathrm{Hopf}}$
for the same reason. In particular, if $C_1,C_2\not\subset\partial\Sigma$,
then $\tSigma$ is a positive $d$-section for the Hopf flow.

If $C_i\subset\partial\Sigma$ for at least one $i\in\{1,2\}$, the
situation is more complicated. A point on $C_i$ has $\alpha_i$
preimages on $C_i^{\alpha_1,\alpha_2}$ under $\rho_{\alpha_1,\alpha_2}$,
whereas each point on $\Sigma\setminus(C_1\cup C_2)$ has
$\alpha_1\alpha_2$ preimages. It follows that along $C_i^{\alpha_1,\alpha_2}$
the lifted surface $\tSigma$ is no longer embedded, but looks
like a rather slim open book with spine $C_i^{\alpha_1,\alpha_2}$
and $\alpha_j$ pages, where $(i,j)$ is a permutation of~$(1,2)$.
We shall describe how to modify the map $\rho_{\alpha_1,\alpha_2}$
to ensure that the lifted surface $\tSigma$ is a $d$-section for
the Hopf flow.
\subsubsection{The behaviour of $\tSigma$ near an exceptional
boundary fibre}
Even though this is not relevant for the genus calculations, it is
worth understanding the behaviour of $\tSigma$ near a boundary fibre
$C_1^{\alpha_1,\alpha_2}$, say. Let $\sigma$ be the curve $\partial V\cap
\Sigma$ as in the proof of Lemma~\ref{lem:1-section} or~\ref{lem:d-section},
with $V$ a tubular neighbourhood of $C_1\subset S^3$. Write
$\tsigma$ for the lift of $\sigma$ to~$\tSigma$. This is a (possibly
disconnected) curve on the boundary of a tubular neighbourhood
$\tV$ of $C_1^{\alpha_1,\alpha_2}\subset S^3_{\alpha_1,\alpha_2}$.

Write $\tmu$ for the meridian on $\partial\tV$, and $\tlambda$
for the longitude determined by the class of a Hopf fibre.
Since $C_1^{\alpha_1,\alpha_2}$ is an $\alpha_1$-fold cover of~$C_1$,
but $\Int(\tSigma)$ is an $\alpha_1\alpha_2$-fold cover of $\Int(\Sigma)$
near the boundary, the curve $\tsigma$ lies in the class
$m\tmu+\alpha_2\tlambda$ for some $m\in\Z$.
The fact that $\tSigma$ is a $d$-section
(away from the boundary) translates into $-d=\tsigma\bullet\tlambda=m$,
cf.\ the proof of Lemma~\ref{lem:d-section}. Thus, if $\alpha_2$
and $d$ are coprime
(which in the global picture is equivalent to $C_2$ being
a boundary fibre, unless $\alpha_2=1$), the curve $\tsigma$ is indivisible,
and hence connected.
Otherwise, by the alternative in \eqref{eqn:alternative}, $\alpha_2$
is a divisor of~$d$, and $\tsigma$ is made up of $\alpha_2$ parallel
curves on~$\partial\tV$.

When we remove thin collar neighbourhoods of $C_1^{\alpha_1,\alpha_2}
\subset\tSigma$ and $C_1\subset\Sigma$, then in the first case we obtain
a covering $\tSigma\rightarrow\Sigma$ with a single boundary component
corresponding to $C_1^{\alpha_1,\alpha_2}$, which is an
$\alpha_1\alpha_2$-fold covering of the corresponding boundary component
of~$\Sigma$. In the second case, $C_1^{\alpha_1,\alpha_2}$ gives
rise to $\alpha_2$ boundary components in~$\tSigma$, each of which is
an $\alpha_1$-fold cover of one and the same boundary component
of~$\Sigma$.

Beware, however, that this modification of $\tSigma$ into a surface with
embedded boundary need not give the correct topology of a positive
$d$-section for the Hopf flow. Before we describe the proper
desingularisation of~$\tSigma$, we analyse the case without
exceptional fibres in the boundary.
\subsubsection{The case $C_i\not\subset\partial\Sigma$, $i=1,2$}
\label{subsubsection:noCi}
When $C_1$ is not contained in the boundary of~$\Sigma$, the
$d$-section $\Sigma$ is near $C_1$ made up of $d/\alpha_1$ discs
transverse to~$C_1$. Every neighbouring regular fibre cuts these discs
$\alpha_1$ times, and hence in a total of $d$ points. Since
$C_1^{\alpha_1,\alpha_2}\rightarrow C_1$ is an $\alpha_1$-fold covering,
the lifted surface $\tSigma$ cuts $C_1^{\alpha_1,\alpha_2}$ in $d$ points,
and the branching index of $\tSigma\rightarrow\Sigma$ at these
points equals~$\alpha_2$.

Similarly, each of the $d/\alpha_2$ intersection points of $\Sigma$
with $C_2$ has $\alpha_2$ preimages in $\tSigma$ of branching
index~$\alpha_1$.

Each of the $k$ boundary components of $\Sigma$ is covered $1:1$ by
$\alpha_1\alpha_2$ boundary components of~$\tSigma$, so the branched
covering $\tSigma\rightarrow\Sigma$ extends to a covering of closed surfaces
without additional branching points.

In \cite[Theorem~4.1]{agz} it was shown that a positive $d$-section
$\tSigma$ for the Hopf flow is a connected, orientable surface of genus
$(d-1)(d-2)/2$ and with $d$ boundary components.
Write $g$ for the genus of~$\Sigma$. Then the Riemann--Hurwitz formula gives
\[ 2-(d-1)(d-2)=\alpha_1\alpha_2\Bigl(2-2g-\frac{d}{\alpha_1}-
\frac{d}{\alpha_2}\Bigr)+2d.\]
With $d=k\alpha_1\alpha_2$ this yields the formula in the first line
of Table~\ref{table:S3}.
\subsubsection{The case $C_i\subset\partial\Sigma$}
We illustrate the situation when one or both of $C_1,C_2$ are
boundary fibres by considering the case $C_1\subset\partial\Sigma$
and $C_2\not\subset\partial\Sigma$. The other cases can be treated
in an analogous fashion.

In order to desingularise the lifted surface $\tSigma$
near~$C_1^{\alpha_1,\alpha_2}$, we modify the copy $S^3_{\alpha_1,\alpha_2}$
of the $3$-sphere and the quotient map~$\rho_{\alpha_1,\alpha_2}$.
For some small $\varepsilon>0$, we define
\[ S^3_{\alpha_1,\alpha_2,\varepsilon}:=
\bigl\{(z_1,z_2)\in\C^2\co|z_1|^{2\alpha_1}+|z_2^{\alpha_2}-\varepsilon
z_1^{\alpha_2}|^2=1\bigr\}.\]
This is still a (strictly) starshaped hypersurface with respect to
the origin, and hence a diffeomorphic copy of the $3$-sphere.

In the commutative diagram of Section~\ref{subsection:alternative},
we replace $S^3_{\alpha_1,\alpha_2}$ by $S^3_{\alpha_1,\alpha_2,\varepsilon}$,
and the map $\rho_{\alpha_1,\alpha_2}$ by
\[ \rho^{\varepsilon}_{\alpha_1,\alpha_2}(z_1,z_2):=
\bigl(z_1^{\alpha_1},z_2^{\alpha_2}-\varepsilon z_1^{\alpha_2}\bigr).\]

The preimage of a point $(w_1,w_2)\in S^3$ under
$\rho^{\varepsilon}_{\alpha_1,\alpha_2}$ consists of the points
\[ (z_1,z_2)=\bigl(w_1^{1/\alpha_1},
(w_2+\varepsilon w_1^{\alpha_2/\alpha_1})^{1/\alpha_2}\bigr),\]
where the choice of root $w_1^{1/\alpha_1}$ is the same in both
components. If $w_1\neq 0$ and $w_2+\varepsilon w_1^{\alpha_2/\alpha_1}\neq 0$
for all choices of $w_1^{1/\alpha_1}$, this gives us $\alpha_1\alpha_2$
preimages. Hence, generically a fibre of $\pi_{1,2}$ is covered $1:1$
by $\alpha_1\alpha_2$ Hopf fibres, as for $\rho_{\alpha_1,\alpha_2}$.

\vspace{2mm}
\paragraph{\textit{(i) The desingularised lift of~$C_1$.}}
While under $\rho_{\alpha_1,\alpha_2}$ the fibre $C_1$ of $\pi_{1,2}$
was covered $\alpha_1$ times by~$C_1^{\alpha_1,\alpha_2}$, and
in particular each point on $C_1$ had only $\alpha_1$ preimages,
under $\rho^{\varepsilon}_{\alpha_1,\alpha_2}$ the point
$(w_1,0)\in C_1$ is covered by $\alpha_1\alpha_2$ points
$\bigl(w_1^{1/\alpha_1},\varepsilon^{1/\alpha_2}w_1^{1/\alpha_1}\bigr)$.
Notice that the $\alpha_1$ different choices of $w_1^{1/\alpha_1}$
yield points on the same Hopf fibre, so the preimage of $C_1$
consists of $\alpha_2$ Hopf fibres in the boundary
of $\tSigma$, each of which is an
$\alpha_1$-fold cover of~$C_1$. This allows us to pass to closed surface
$\hat{\tSigma},\hat{\Sigma}$ as in the proof of Theorem~\ref{thm:d-section},
and the boundary component $C_1\subset\partial\Sigma$ contributes
$\alpha_2$ points in $\hat{\tSigma}$ of branching index $\alpha_1$
sitting over a single point in~$\hat{\Sigma}$.

\vspace{2mm}
\paragraph{\textit{(ii) Branch points with $z_1=0$.}}
One case when a point in $S^3$ has fewer than $\alpha_1\alpha_2$
preimages is when $w_1=0$, that is, points on~$C_2$. The fibre
$C_2$ of $\pi_{1,2}$ is covered $\alpha_2$ times by the Hopf
fibre~$C_2^{\alpha_1,\alpha_2}$. Since $C_2$ is a fibre of
multiplicity~$\alpha_2$, the $d$-section $\Sigma$ cuts it in $d/\alpha_2$
points. As in Section~\ref{subsubsection:noCi} this gives us $d$
branch points upstairs on $C_2^{\alpha_1,\alpha_2}$ of branching
index~$\alpha_1$.

\vspace{2mm}
\paragraph{\textit{(iii) Branch points with $z_2=0$.}}
The only other case when a point $(w_1,w_2)\in S^3$ has fewer than
$\alpha_1\alpha_2$ preimages under $\rho^{\varepsilon}_{\alpha_1,\alpha_2}$
is when $w_2+\varepsilon w_1^{\alpha_2/\alpha_1}=0$ for some
choice of $w_1^{1/\alpha_1}$. These correspond to points upstairs
with $z_2=0$, that is, points on~$C_1^{\alpha_1,\alpha_2}$.
The restriction of $\rho^{\varepsilon}_{\alpha_1,\alpha_2}$ to
$C_1^{\alpha_1,\alpha_2}$ is given by
\[ \bigl(\rme^{\rmi\theta},0\bigr)\longmapsto\bigl(\rme^{\rmi\alpha_1\theta},
-\varepsilon\rme^{\rmi\alpha_2\theta}\bigr),\]
so $C_1^{\alpha_1,\alpha_2}$ maps injectively to the $\pi_{1,2}$-fibre
through the point $(1,-\varepsilon)$. We choose $\varepsilon>0$
small enough to ensure that this is not a boundary component of~$\Sigma$.

For ease of notation we set $\lambda=\rme^{\rmi\theta}$. We now ask: what
are the preimages of $(w_1,w_2)=(\lambda^{\alpha_1},-\varepsilon
\lambda^{\alpha_2})$, except for the point $(\lambda,0)\in
C_1^{\alpha_1,\alpha_2}$? The equations
\[ \left\{\begin{array}{l}
z_1^{\alpha_1}=\lambda^{\alpha_1},\\
z_2^{\alpha_2}-\varepsilon z_1^{\alpha_2}=-\varepsilon\lambda^{\alpha_2}
\end{array}\right. \]
have the solutions
\[ \left\{\begin{array}{l}
z_1=\lambda\rme^{2\pi\rmi k/\alpha_1},\;\; k\in\{0,\ldots,\alpha_1-1\}\\
z_2=\sqrt[\alpha_2]{\varepsilon}\lambda
\bigl(\rme^{2\pi\rmi k/\alpha_1}-1\bigr)^{1/\alpha_2},
\end{array}\right. \]
where $\sqrt[\alpha_2]{\varepsilon}$ denotes the root in~$\R^+$.
For each $k\in\{1,\ldots,\alpha_1-1\}$ (\emph{sic}!) and given $\lambda$
we obtain $\alpha_2$ solutions on different Hopf fibres. For $k=0$
we obtain the unique solution $(\lambda,0)$, where we have branching
of index~$\alpha_2$.

The $d$-section $\Sigma$ cuts the $\pi_{1,2}$-fibre through
the point $(1,-\varepsilon)$ in $d$ points. For each of these $d$ points we
obtain a single branch point upstairs of index~$\alpha_2$.

\vspace{2mm}
\paragraph{\textit{Computing the genus of $\tSigma$.}}
We can now compute the genus of $\tSigma$ by counting the
branch points of $\tSigma\rightarrow\Sigma$ upstairs. We have
\begin{itemize}
\item[(i)] $\alpha_2$ branch points of index $\alpha_1$,
\item[(ii)] $d$ branch points of index $\alpha_1$,
\item[(iii)] $d$ branch points of index $\alpha_2$.
\end{itemize}
The Riemann--Hurwitz formula then tells us that
\[ 2-(d-1)(d-2)+\alpha_2(\alpha_1-1)+d(\alpha_1-1)+d(\alpha_2-1)=
\alpha_1\alpha_2(2-2g).\]
With $d=k\alpha_1\alpha_2+\alpha_2$ we obtain the formula in the second line
of Table~\ref{table:S3}.

\begin{rem}
In the case when $C_1\not\subset\partial\Sigma$ and $C_2\subset\partial\Sigma$,
one simply reverses the roles of $z_1$ and $z_2$ in the above argument.
If both $C_1$ and $C_2$ are boundary fibres, one replaces
$S^3_{\alpha_1,\alpha_2}$ by
\[ S^3_{\alpha_1,\alpha_2,\varepsilon_1,\varepsilon_2}:=
\bigl\{(z_1,z_2)\in\C^2\co|z_1^{\alpha_1}-\varepsilon_1z_2^{\alpha_1}|^2
+|z_2^{\alpha_2}-\varepsilon_2
z_1^{\alpha_2}|^2=1\bigr\}\]
for some small $\varepsilon_1,\varepsilon_2>0$, and $\rho_{\alpha_1,\alpha_2}$
by
\[ \rho^{\varepsilon_1,\varepsilon_2}_{\alpha_1,\alpha_2}(z_1,z_2):=
\bigl(z_1^{\alpha_1}-\varepsilon_1z_2^{\alpha_1},
z_2^{\alpha_2}-\varepsilon_2 z_1^{\alpha_2}\bigr).\]
\end{rem}
\section{Surfaces of section and algebraic curves}
\label{section:algebraic}
We now wish to relate surfaces of section for Seifert fibrations of $S^3$
to algebraic curves in weighted projective planes
$\bbP(1,\alpha_1,\alpha_2)$.
The degree-genus formula for such algebraic curves~\cite{orwa72}
then allows us to confirm Corollary~\ref{cor:S3}. Conversely,
our results on positive $d$-sections can be read as an alternative
proof of the degree-genus formula.
Moreover, the lifting of $d$-sections
as described in Section~\ref{subsection:lifting} permits a natural
interpretation as the lifting of algebraic curves from
$\bbP(1,\alpha_1,\alpha_2)$ to~$\CP^2$.

In \cite{agz} it was explained how an algebraic curve $C\subset\CP^2$
of degree $d$ gives rise to a surface of section $\Sigma\subset S^3$ for
the Hopf flow. Simply consider the affine part $\Ca:=C\cap\C^2$ of $C$,
and project $\Ca\setminus\{(0,0)\}$ radially to the unit sphere
$S^3\subset\C^2$. Under suitable assumptions on~$C$, this will
be a surface of section $\Sigma$ for the Hopf flow, with positive boundary
components coming from points on $C$ at infinity, and negative
boundary components if $(0,0)\in\Ca$. The multiplicity of the section
is determined by the number of intersection points of $\Ca\setminus\{(0,0)\}$
with a radial plane in~$\C^2$.

Topologically, the algebraic curve is obtained from the
corresponding surface of section by capping off the
boundary components with discs. This means that the genus of the
algebraic curve equals that of the surface of section.

For instance, if $C=\{F=0\}$ with $F$ a homogeneous polynomial
of degree $d$ of the form
\[ F(z_0,z_1,z_2)=f(z_1,z_2)-z_0^d,\]
the intersection of $\Ca$ with the punctured
radial plane $P_{a,b}:=\{(az,bz)\co z\in\C^*\}$,
where $(a,b)\in S^3\subset\C^2$,
is given by the solutions $z$ to the
equation $f(a,b)z^d=1$. If $f(a,b)\neq 0$, this gives rise to
$d$ points $(az,bz)\in\Ca\cap P_{a,b}$ whose projection to $S^3$
lie on the Hopf fibre
through the point $(a,b)$.
If $f(a,b)=0$, there are no intersections of $\Ca$ with the radial
plane, but now the point $[0:a:b]$ at infinity lies on $C$ and corresponds
to a boundary fibre of $\Sigma$. The surface of section $\Sigma$
has no negative boundary components, since $F(1,0,0)\neq 0$.

It is this simple example that we want to generalise to
positive $d$-sections of $\eqref{eqn:action12}$ and algebraic curves
in weighted projective planes,
as it illustrates some of the essential features of the
correspondence between these two classes of objects.
Moreover, we shall see (Remark~\ref{rem:every-section})
that every positive $d$-section can be
realised by an algebraic curve of this form.
\subsection{Weighted homogeneous polynomials}
As explained in Section~\ref{subsection:alg-curves},
an algebraic curve in $\bbP(1,\alpha_1,\alpha_2)$ is defined by
a $(1,\alpha_1,\alpha_2)$-weighted homogeneous polynomial.
We want to consider such polynomials of the form
\begin{equation}
\label{eqn:F}
F(z_0,z_1,z_2)=f(z_1,z_2)-z_0^d,
\end{equation}
The following lemma is the analogue of the factorisation
of an unweighted homogeneous polynomial in two variables
into linear factors~\cite[Lemma~2.8]{kirw92}.

\begin{lem}
\label{lem:poly}
If $F$ is non-singular, then $f$ is of the form
\[ f(z_1,z_2)=(b_1^{\alpha_1}z_1^{\alpha_2}-a_1^{\alpha_2}z_2^{\alpha_1})
\cdot\ldots\cdot(b_k^{\alpha_1}z_1^{\alpha_2}-a_k^{\alpha_2}z_2^{\alpha_1})
\cdot z_1^{\varepsilon_1}\cdot z_2^{\varepsilon_2}\]
with $\varepsilon_1,\varepsilon_2\in\{0,1\}$ and
$[a_i:b_i]_{(\alpha_1,\alpha_2)}\in\bbP(\alpha_1,\alpha_2)$,
$i=1,\ldots,k$, pairwise distinct points different from
$[1:0]_{(\alpha_1,\alpha_2)}$ and $[0:1]_{(\alpha_1,\alpha_2)}$. If at
least one $\varepsilon_i$ equals~$1$, one may have $k=0$.

In particular, the degree $d$ of $F$ equals
$d=k\alpha_1\alpha_2+\varepsilon_1\alpha_1+\varepsilon_2\alpha_2$.
\end{lem}

\begin{proof}
We have $f(0,1)=0$ if and only if $f$ is divisible by~$z_1$. Similarly,
$f(1,0)=0$ holds if and only if $f$ is divisible by~$z_2$.
In either case, there can be at most one such factor each of
$z_1$ or $z_2$, lest $F$ be singular.

After dividing by these factors, in case they are present, we have
a sum of complex monomials $c_{mn}z_1^mz_2^n$ with
$\alpha_1m+\alpha_2n=\mathrm{const.}$\ and at least one non-zero monomial
each without a $z_1$- or $z_2$-factor, respectively.
From the identity $\alpha_1m+\alpha_2n=\alpha_1m'+\alpha_2n'$
we deduce $\alpha_1|(n-n')$ and $\alpha_2|(m-m')$, which implies that
$f$ is of the form
\[ f(z_1,z_2)=z_1^{\varepsilon_1}\cdot z_2^{\varepsilon_2}\cdot
\sum_{i=0}^k c_iz_1^{i\alpha_2}z_2^{(k-i)\alpha_1} \]
with $c_0,c_k\neq 0$.

The sum in this description is a homogeneous polynomial
of degree $k$ in $z_1^{\alpha_2}$ and $z_2^{\alpha_1}$,
and hence, by \cite[Lemma~2.8]{kirw92}, has a factorisation
into linear factors in $z_1^{\alpha_2}$ and $z_2^{\alpha_1}$.
The condition on
the $[a_i:b_i]_{(\alpha_1,\alpha_2)}$ is equivalent to
$f$ having no repeated factors.
\end{proof}
\subsection{Constructing $d$-sections from algebraic curves}
We now want to show how a complex algebraic curve $C=\{F=0\}$,
with $F$ as in~\eqref{eqn:F}, gives rise to a positive
$d$-section of the $A_{1,2}$-flow. For this we need to work
with a weighted `radial' projection $\C^2\setminus\{(0,0)\}
\rightarrow S^3$
given by sending $(z_1,z_2)$ to the unique point
on $S^3$ of the form $(z_1/r^{\alpha_1},z_2/r^{\alpha_2})$, $r\in\R^+$.
This projection map is smooth. We define weighted `punctured
radial complex planes'
for $(a,b)\in S^3\subset\C^2$ as
\[ P_{a,b}:=\bigl\{(az^{\alpha_1},bz^{\alpha_2})\co z\in\C^*\bigr\}.\]
Under the radial projection we have, for $z=r\rme^{\rmi\theta}$,
\[ \bigl(az^{\alpha_1},bz^{\alpha_2}\bigr)\longmapsto
\bigl(\rme^{\rmi\alpha_1\theta}a,
\rme^{\rmi\alpha_2\theta}b\bigr),\]
that is, $P_{a,b}$ projects to the $A_{1,2}$-orbit through $(a,b)$
on~$S^3$.

The following is the generalisation of \cite[Proposition~5.1]{agz}
from the Hopf flow to all Seifert fibrations of~$S^3$.

\begin{prop}
The complex algebraic curve $C=\{F=0\}$ in the weighted complex plane
$\bbP(1,\alpha_1,\alpha_2)$, where $F$ is a weighted
homogeneous polynomial of the form~\eqref{eqn:F},
defines a positive $d$-section for the flow $(z_1,z_2)\mapsto
\bigl(\rme^{\rmi\alpha_1\theta}z_1,\rme^{\rmi\alpha_2\theta}z_2\bigr)$
on $S^3$ if and only if $F$ is non-singular.
\end{prop}

The proof of this proposition is analogous to that of
\cite[Proposition~5.1]{agz}, and we only sketch the necessary modifications.

The affine part $\Ca\bigl\{(z_1,z_2)\in\C^2\co f(z_1,z_2)=1\bigr\}$
can be shown to intersect each $P_{a,b}$ in $d$ points
that do not lie on a real radial ray, and hence project to $d$ distinct
points on the corresponding $A_{1,2}$-orbit on~$S^3$, except in the following
cases.

\begin{itemize}
\item[-] Solutions of the equation $f(z_1,0)=1$ or $f(0,z_2)=1$
correspond to $C_1$ or $C_2$ \emph{not} being a boundary fibre,
and they give rise to $d/\alpha_1$ or $d/\alpha_2$ intersection
points, respectively.
\item[-] Solutions to the equation $f(z_1,z_2)=0$, that is,
points at infinity on~$C$, correspond to boundary fibres.
\end{itemize}

The correct asymptotic behaviour near the boundary fibres
can be demonstrated with the explicit factorisation of $f$
given by Lemma~\ref{lem:poly}.
\subsection{Branched coverings of algebraic curves}
\label{subsection:coverings-curves}
The map $\rho_{\alpha_1,\alpha_2}\co S^3_{\alpha_1,\alpha_2}\rightarrow
S^3$ has a natural extension to
\[ \begin{array}{rccc}
\rho_{\alpha_1,\alpha_2}\co
 & \CP^2         & \longrightarrow & \bbP(1,\alpha_1,\alpha_2)\\
 & [z_0:z_1:z_2] & \longmapsto     & [z_0:z_1^{\alpha_1}:
                                      z_2^{\alpha_2}]_{(1,\alpha_1,\alpha_2)}.
\end{array}\]
A weighted algebraic curve $C=\{F=0\}\subset\bbP(1,\alpha_1,\alpha_2)$
lifts to the algebraic curve $\tC=\{\tF=0\}\subset\CP^2$, where
$\tF$ is the homogeneous polynomial of degree $d$ defined by
$\tF(z_0,z_1,z_2)=F(z_0,z_1^{\alpha_1},z_2^{\alpha_2})$.
The lift $\tC$ of a non-singular algebraic curve  $C$
may well be singular.

Under this extended map $\rho_{\alpha_1,\alpha_2}$,
the standard punctured plane
\[ \bigl\{(az,bz)\co z\in\C^*\bigr\}\subset\C^2, \]
with $(a,b)\in S^3_{\alpha_1,\alpha_2}$, maps
to~$P_{a^{\alpha_1},b^{\alpha_2}}$.

The discussion in Section~\ref{subsection:lifting}
carries over to algebraic curves $C=\{F=0\}$ with
$F$ as in~\eqref{eqn:F}, only changing notation to (weighted)
homogeneous coordinates.
This means that we can compute the
genus of a weighted algebraic curve in $\bbP(1,\alpha_1,\alpha_2)$
by lifting it to an algebraic curve in $\CP^2$, and then
applying the Riemann--Hurwitz formula and the known results
about algebraic curves in $\CP^2$.
In the covering $\tC\rightarrow C$
we see the same branching behaviour
as for the (extended) surfaces of section, and
we can apply
the same process of desingularisation in case $C$ contains
one or both of the points $[0:1:0]_{(1,\alpha_1,\alpha_2)}$
or $[0:0:1]_{(1,\alpha_1,\alpha_2)}$ at infinity.
\subsection{Computing the genus of algebraic curves via
surfaces of section}
Here is the general formula for computing the genus of a
non-singular algebraic curve in a weighted projective plane,
see~\cite[Corollary~3.5]{orwa72}, \cite[Theorem~5.3.7]{hosg20}.

\begin{thm}
\label{thm:d-g}
The genus of a non-singular algebraic curve of degree $d$ in
the weighted projective plane $\bbP(a_0,a_1,a_2)$ is given by
\[ g=\frac{1}{2}\Bigl(\frac{d^2}{a_0a_1a_2}-
d\sum_{i<j}\frac{\gcd(a_i,a_j)}{a_ia_j}+
\sum_i\frac{\gcd(a_i,d)}{a_i}-1\Bigr).\qed\]
\end{thm}

For $(a_0,a_1,a_2)=(1,\alpha_1,\alpha_2)$ this confirms the
values in Table~\ref{table:S3}.

\begin{rem}
\label{rem:every-section}
Consider the possible values of $d$ for which there exists,
according to Table~\ref{table:S3}, a positive $d$-section
for the Seifert fibration of $S^3$ defined by~\eqref{eqn:action12}.
By Lemma~\ref{lem:poly}, any such $d$ can be realised as the
degree of a weighted homogeneous polynomial $F$ of the form~\eqref{eqn:F}.
Hence, the algebraic curve $\{F=0\}$ in $\bbP(1,\alpha_1,\alpha_2)$
gives rise to a $d$-section with the chosen value of~$d$.
By the uniqueness statement in Theorem~\ref{thm:d-section},
this means that \emph{every} surface of section for any Seifert fibration
of $S^3$ comes from an algebraic curve $\{F=0\}\subset
\bbP(1,\alpha_1,\alpha_2)$ with $F$ as in~\eqref{eqn:F}.
\end{rem}

Conversely, the arguments in Section~\ref{subsection:lifting}
--- applied to coverings of algebraic surfaces as
explained in Section~\ref{subsection:coverings-curves} ---
provide an alternative proof of Theorem~\ref{thm:d-g}
in the case $a_0=1$ and for such special algebraic curves; the latter
restriction can be removed with genericity arguments as in
\cite[Section~6]{agz}.

It seems plausible that the reasoning in Section~\ref{subsection:lifting}
can be extended to a geometric proof of Theorem~\ref{thm:d-g}
in full generality.
\begin{ack}
We are grateful to Christian Lange and Umberto Hryniewicz
for their interest in this work. Their questions and comments
have contributed significantly to the writing of this paper.
The anonymous referee made useful suggestions for improving the
exposition.
\end{ack}
\end{document}